\documentclass{amsart} 
\usepackage{amsmath,amscd} 
\usepackage{latexsym,amsmath,amssymb,mathrsfs,setspace,enumerate}
\usepackage[T1]{fontenc}
\usepackage[utf8]{inputenc}
\usepackage{lmodern} 
\usepackage{amssymb}
\usepackage{enumitem}
\usepackage{tikz-cd}
\usepackage{setspace}
\usepackage{graphics}
\usepackage{graphicx}
\usepackage[utf8]{inputenc}
\usepackage[english]{babel}
\usepackage{fancyhdr}
\usepackage{tikz}
\usetikzlibrary{decorations.markings,positioning}
\theoremstyle{plain}

\newtheorem{prop}{Proposition} 

\newtheorem{exam}{Example}

\newtheorem{defn}{Definition}

\tikzcdset{%
	triple line/.code={\tikzset{%
			double distance = 3pt, 
			double=\pgfkeysvalueof{/tikz/commutative diagrams/background color}}},
	quadruple line/.code={\tikzset{%
			double distance = 5.3pt, 
			double=\pgfkeysvalueof{/tikz/commutative diagrams/background color}}},
	Rrightarrow/.code={\tikzcdset{triple line}\pgfsetarrows{tikzcd implies cap-tikzcd implies}},
	RRightarrow/.code={\tikzcdset{quadruple line}\pgfsetarrows{tikzcd implies cap-tikzcd implies}}
}

\title{Categories of bundles and categories of chains}
\author {P. G. Romeo and Riya Jose }
\address{Dept. of Mathematics, Cochin University of Science and Technology, Kochi, Kerala, INDIA.}
\email{$romeo_-parackal@yahoo.com,\, riyajosemarangattu@gmail.com $}
\thanks{} 
\date{}

\begin{document}
	
	\begin{abstract}
		K. S. S. Nambooripad introduced an interesting class of categories known as normal categories, which are categories with subobjects,  morphisms admitting factorization and having sufficiently many cones. These normal categories plays fundamental role in the study of structure of regular semigroups.
		In \cite{PGR} we discussed the category of chain bundles and category of chains.
		In the present paper revisits classical notions of bundles, including fibre bundles, vector bundles, and principal G-bundles, and discuss the chain categories arising from the category of bundles. Moreover, its is verified that these categories are categories with subobjects.  
		\end{abstract}
	\maketitle
	
Category theory provides a powerful and unifying language for studying mathematical structures and the relationships between them.  A significant development in this direction was made by K. S. S. Nambooripad, who introduced the notion of normal categories, which are categories with subobjects, morphisms admitting factorization and having sufficently many cones. 
The presence of subobjects, together with suitable factorization properties, allows for a deeper analysis of internal organization within categories and has important applications in the study of algebraic systems. In \cite{PGR}, we introduced chains as a sequence of objects in the category $\mathcal{C}$ such that for any two objects $X$ and $Y $ any map in $Hom(X, Y )$ constitute morphism between $X$ and $Y$. A category whose objects are such chains and whose morphisms are appropriate chain maps is called the category of chains. A bundle is a triple $(E,p,B)$ where $E$ and $B$ are spaces and $p:E \rightarrow B $ is a map. In this paper, we study the category of bundles and the associated chain categories arising from them. Furthermore, we verify whether  these categories admit subobjects.

	\section{Preliminaries}
	
	In the following, we briefly recall some basic notions related to categories, subobjects, and preorder.
	
	A category is mathematical structure consisting of objects (vertices) and morphisms (arrows) between them that satisfies a partial composition in which each object has an identity morphism that acts neutrally in composition.
	A structure preserving map between two categories is called a functor.
	A category $\mathcal{D}$ is a subcategory of a category $\mathcal{C}$ if the class $\mathcal{D}$ is a subclass of $\mathcal{C}$ and the composition in $\mathcal{D}$ is the restriction of the composition in $\mathcal{C}$ to $\mathcal{D}$. In this case, the inclusion $\mathcal{D} \subseteq \mathcal{C}$ preserves composition and identities and so represents a functor of $\mathcal{D}$ to $\mathcal{C}$.

A morphism $f$ in a category $\mathcal{C}$ is a monomorphism if 
$$ gf = hf \Rightarrow g = h \ \forall g,h \in \mathcal{C} $$
Every morphism in a concrete category whose underlying function is an injection is a monomorphism. A morphism $f$ in a category $\mathcal{C}$ is an \textit{epimorphism} if 
$$ fg = fh \Rightarrow g = h \ \forall g,h \in \mathcal{C} $$
Every morphism in a concrete category whose underlying function is an surjection is an epimorphism.

Let $\mathbb{M} \mathcal{C}$ denote the class of all monomorphisms in $\mathcal{C}$. On  $\mathbb{M} \mathcal{C}$, define the relation $$f \preceq g \Leftrightarrow f = hg \quad \text{for some } h \in \mathcal{C}$$
$\preceq$ is a quasi- order and  
\begin{equation}
	\sim = \preceq \cap \preceq^{-1}
\end{equation} is an equivalence relation on $\mathbb{M} \mathcal{C}$.

\begin{defn}
	A preorder $\mathcal{P}$ is a category such that for any $p, p' \in  \mathcal{P}$, $\mathcal{P}(p,p')$ contains at most one morphism. In this case, the relation $\subseteq$ on the class $\nu \mathcal{P}$ defined by 
	\begin{equation}
		p \subseteq p' \ \Leftrightarrow \ \mathcal{P}(p,p') \neq \emptyset
	\end{equation}
	is a quasiorder. When $\mathcal{P}$ is a preorder, $\nu \mathcal{P}$ will stand for the quasiordered class $(\nu \mathcal{P}, \subseteq)$. Conversely given a quasiorder $\leq$ on the class $X$, the subset $$ \mathcal{P} = \{(x,y) \in X \times X : x \leq y\}$$
	of $ X \times X $ is a preorder such that the quasiordered class $\nu \mathcal{P}$ defined above is order isomorphic with $(X,\leq)$. If the relation $\subseteq$ on $\mathcal{P}$  is antisymmetric then we shall say that $\mathcal{P}$ is a \textit{strict preorder}.
\end{defn}

\begin{defn} 
	Let $ \mathcal{C} $ be a category. A choice of subobjects in $ \mathcal{C} $ is a subcategory $ \mathcal{P} \subseteq \mathcal{C} $ satisfying the following:
	\begin{enumerate}
		\item[(a)] $\mathcal{P}$ is a strict preorder with $\nu \mathcal{P} = \nu \mathcal{C}$.
		\item[(b)] Every $f \in \mathcal{P} $ is a monomorphism in  $\mathcal{C}$.
		\item[(c)] If $f, g \in \mathcal{P}$ and if $f = hg$ for some $h \in \mathcal{C}$ then $h \in \mathcal{P}$.
	\end{enumerate}
	When $\mathcal{P}$ satisfies these conditions, the pair $ (\mathcal{C},\mathcal{P})$ is called a category with subobjects.
\end{defn}
When $\mathcal{C} $ has  subobjects, unless explicitly stated otherwise, $ \nu\mathcal{C} $ will denote the choice of subobjects in $ \mathcal{C} $. The partial order defined by equation $(2)$ is called the preorder of inclusions or subobject relation in $ \mathcal{C} $ and is denoted by $ \subseteq $. If $c,d \in \nu\mathcal{C}$ and $ c\subseteq d $ the unique morphism from $c$ to $d$ is the inclusion $j_c^d: c \rightarrow d $. 

\begin{exam}
	In categories $Set, Grp, Vect_K, Mod_R $ the relation on objects induced  by the usual set inclusion is a subobject relation.
\end{exam}

	\section{Category of Bundles}
	In this section we recall concepts of bundles with examples and the category of bundles, for a detailed discussion see \cite{Dale}. We see that category of bundles is a category with subobjects.
	
	\begin{defn}
		A bundle is a triple $(E,p,B)$ where $p: E \rightarrow B$ is a map. The space $B$ is called the base space, the space $E$ is called the total space, and the map $p$ is called the projection of the bundle. For each $b \in B$, the space $p^{-1} (b) $ is called the fibre of the bundle over $b \in B$. 
		\end{defn}

	\begin{exam}
		The product bundle over $B$ with fibre $F$ is $(B \times F, p,B)$, where $p$ is the projection on the first factor.
	\end{exam}

		A $k-$ dimensional vector bundle $\xi$ over $\textbf{F}$ is a bundle $(E,p,B)$ together with the structure of a $k-$ dimensional vector space over $\textbf{F}$ on each fibre $p^{-1} (b) $ such that the following local triviality condition is satisfied. Each point of $B$ has an open neighbourhood $U$ and a  $U-$ isomorphism $h: U \times F^k \rightarrow p^{-1}(U)$ such that the restriction $ b \times F^k \rightarrow p^{-1} (b)$ is a vector space isomorphism for each $b \in U$.
		
	\begin{exam}
		The tangent bundle over $S^n$, denoted $\tau (S^n) = (T, p, S^n)$, and the normal bundle over $S^n$, denoted $\nu (S^n) = (N, q, S^n)$, are two subbundles of the product bundle $(S^n \times \mathbb{R}^{n+1}, p, S^n)$ whose total spaces are defined by the relation $(b,x) \in T$ if and only if the inner product $(b|x) = 0$ and by $(b,x) \in N$ if and only if   $x = kb$ for some $k \in \mathbb{R}$. 
		An element $(b,x) \in T$ is called a tangent vector to $S^n$ at $b$, and an  element $(b,x) \in N$ is called a normal vector to $S^n$ at $b$. The fibres $p^{-1}(b) \subset T $ and $q^{-1}(b) \subset N $ are vector spaces of dimension $n$ and $1$, respectively.
	\end{exam}

 In the following we recall the definition of a principal $G$- bundle.
	\begin{defn} 
		Let $G$ be a topological group. A principal $G-$ bundle is a map $p: E \rightarrow B$ together with a right action $r: E \times G \rightarrow E$ satisfying the following conditions:\\
		$ 1)\, $ For every $ x \in E $ and $g \in G$, $p(xg) = p(x)$. i.e., $p$ induces a homeomorphism between quotient spaces of this action and $B$.\\
		$2)\, $ For every $b \in B, \, \exists $ an open neighbourhood $U$ and a $G- $ homeomorphism $\phi : p^{-1}(U) \rightarrow U \times G$ (where $G$ acts on $p^{-1}(U)$ by restriction of $r$ on $U \times G$ by $((u,x),g)\longmapsto (u, xg)$) such that following diagram commutes:
		\begin{center}  
			\begin{tikzcd} 
				p{^{-1}(U)} \arrow["\phi "]{r} \arrow["p"]{d} & U \times G  \arrow{dl}{Pr_U}
				\\
				U \\
			\end{tikzcd}		
		\end{center}
	\end{defn}
	\begin{exam}
		Let $B$ be the unit circle $S^1$ parameterized by the angle $\theta$. The total space $E$ is simply the product space $S^1 \times G$, where 
		$G$ is any Lie group. And for simplicity, let's take $G = \mathbb{R}$. The projection map $p : E \rightarrow B$ is defined as $p(\theta, g) = \theta$, where $(\theta, g) \in S^1 \times \mathbb{R}$. The fibre $F$ over each point $b\in B$ is the copy of 
		$G$ at that point.
		In this case, the fibre $F$ over each point 
		$ \theta \in S^1$ is just $G = \mathbb{R}$. The action of the Lie group $G = \mathbb{R}$ on the fibre $F = \mathbb{R}$ is given by addition: $g.f = g + f $, for $g, f \in \mathbb{R}$.\\
		This bundle is called trivial because the total space $E$ is simply the product space $S^1 \times \mathbb{R}$, and each fibre is isomorphic to the structure group $G = \mathbb{R}$. The projection map $p$ is also straightforward, projecting each point to its first component on the base space $S^1$.
	\end{exam}
	\begin{exam}
		The frame bundle of a smooth $n$-dimensional manifold $M$, is a principal $GL(n,\mathbb{R})$-bundle. It consists of all ordered bases (frames) of the tangent spaces of $M$.
		Here the total space is\\ $E = \underset{x\in M}{\cup}\{(x, e_1, e_2, \cdots, e_n)| \{e_1\}_{i=1}^n \textnormal{is a basis of tangent space} \hspace{2mm} T_xM\}$.\\
		Base space is $B = M$. The general linear group \, which represents the set of all changes of basis in $\mathbb{R}^n$ is the fibre. The projection $\pi : E \rightarrow B$ is given by $\pi(x, e_1, e_2, \cdots, e_n)=x$. The general linear group GL(n,R) acts freely and transitively on the fibres. 
	\end{exam}

	A space $\textbf{F}$ is the fibre of a bundle  $(E, p, B)$ if every $p^{-1}(b)$ for $b \in B$ is homeomorphic to $\textbf{F}$.  
	Product bundles, tangent bundles, frame bundles discussed above are examples whose fibres are $F$, vector spaces, general linear groups respectively.

	\begin{defn}
		Let $(E, p, B)$ and $(E',p',B')$ be two bundles. A bundle morphism $(u,f) : (E, p, B) \rightarrow (E',p',B')$ is a pair of maps $u: E \rightarrow E'$ and $f: B \rightarrow B'$ such that $up' = pf$.
	\end{defn}
	For two bundles over same base $B$, ie., for  $(E, p, B)$ and $(E',p',B)$ a bundle morphism over $B$ (or $B-$ morphism) $u : (E, p, B) \rightarrow (E',p',B)$ is a map $u: E \rightarrow E'$ such that $up' = p$.	
	In case of morphism of vector bundles the restriction $u: p^{-1}(b) \rightarrow (p')^{-1}(f(b))$ is linear for each $b \in B$ and for $B-$ morphism the restriction $u: p^{-1}(b) \rightarrow (p')^{-1}(b)$ is linear for each $b \in B$.\\
	
	Now, consider a family of bundles $$\{(E, p, B) : E,B \,\text{are total and base spaces, $p$ bundle projection}\}$$
	with the set of all bundle morphisms $(u,f) : (E, p, B) \rightarrow (E',p',B')$ with composition of morphisms  $(u,f) : (E, p, B) \rightarrow (E',p',B')$ and $(u',f') : (E',p',B') \rightarrow (E'',p'',B'')$ 
	 is given by the following commutative diagram: 
	\begin{center}
		\begin{tikzcd}
			E  \arrow["u"]{r}\arrow{d}{p} & E' \arrow["u'"]{r}\arrow{d}{p'} & E'' \arrow{d}{p''}\\
			B  \arrow["f"]{r} & B' \arrow["f'"]{r} & B''\\
		\end{tikzcd}
	\end{center}
	Consequently, composite is the bundle morphism $(uu', ff'):  (E, p, B) \rightarrow (E'',p'',B'')$. The pair $(1_E,1_B):(E, p, B) \rightarrow (E, p, B)$ is a bundle morphism and is the identity morphism. It can be easily seen that all axioms of a category are satisfied.  
	 \begin{prop}
	 	All bundles of the form  $(E, p, B)$ and bundle morphisms forms a category called category of bundles, denoted \textbf{Bun}.
	 \end{prop}
 A bundle morphism $(u,f) : (E, p, B) \rightarrow (E',p',B')$  is an isomorphism if and only if there exists a morphism  $(u',f') : (E',p',B') \rightarrow (E, p, B)$ such that $f'f = 1_B, ff' = 1_B, u'u = 1_E$ and $uu' = 1_E$.
	
	\begin{defn}(cf.\cite{Dale}) \label{subbundle}
	A bundle $(E',p',B')$ is a subbundle of $(E,p,B)$ provided $E'$ is a subspace of $E$, $B'$ is a subspace of $B$, and $p' = p|E' : E' \rightarrow B'$.
\end{defn}
\begin{center}
	\begin{tikzcd}
		E'  \arrow["p'"]{r}\arrow{d}{j_{E'}^{E}} & B'\arrow{d}{j_{B'}^{B}} \\
		E  \arrow["p"]{r} & B \\
	\end{tikzcd}
\end{center}
The bundle morphism $(j_{E'}^{E}, j_{B'}^{B})$ is the inclusion. Being a subbundle is a strict preorder on the vertex class $\nu \textbf{Bun}$.
\begin{prop}
	 The category of bundles \textbf{Bun} with being a subbundle as the choice of subobjects $\mathcal {P}$, is a category with subobjects (\textbf{Bun}, $\mathcal {P}$).
\end{prop}
\begin{proof} Since every bundle is a subbundle of itself and $\mathcal {P}$ being the choice of subbundles we have $\nu \mathcal {P} = \nu \textbf{Bun}$ and morphisms in $\mathcal {P}$  are limited to the canonical inclusions [ see Definition \ref{subbundle}].
	 Then $(\textbf{Bun},\mathcal {P} )$ is a category with subobjects, since 
	\begin{enumerate}
		\item $\mathcal {P}$ forms a strict preorder with $\nu \mathcal {P} = \nu \textbf{Bun}$.
		\item Every morphism  $(j_{E'}^{E}, j_{B'}^{B}) \in \mathcal {P}$ is a monomorphism being pair of inclusions.
		\item If $ (u, f) = (u'', f'')(u', f') $ where $ (u, f), (u', f') \in  \mathcal {P}$, then $(u'', f'') \in  \mathcal {P}$.\\
		    For, 
		    \begin{center}
		    	\begin{tikzcd}
		    		E''  \arrow["p''"]{r}\arrow{d}{u''} & B''\arrow{d}{f''} \\
		    		E'  \arrow["p'"]{r}\arrow{d}{j_{E'}^{E}} & B'\arrow{d}{j_{B'}^{B}} \\
		    		E  \arrow["p"]{r} & B \\
		    	\end{tikzcd}
		    \end{center}
	    \begin{center}
	    	\begin{tikzcd}
	    		E''  \arrow["p''"]{r}\arrow{d}{j_{E''}^{E}} & B''\arrow{d}{j_{B''}^{B}} \\
	    		E  \arrow["p"]{r} & B \\
	    	\end{tikzcd}
	    \end{center}
    If $ (u, f) = (u'', f'')(u', f') $, then we have $j_{E''}^{E} = u''j_{E'}^{E}$ and $j_{B''}^{B} = f'' j_{B'}^{B}$.\\
    $j_{E''}^{E} = u''j_{E'}^{E} \implies j_{E''}^{E'}j_{E'}^{E} = u''j_{E'}^{E} \implies j_{E''}^{E'} = u''$. Similarly we get, $ j_{B''}^{B'} = f''$. Thus, $(E'',p'',B'')$ becomes a  subbundle of $(E', p', B')$ and hence $(u'', f'') = (j_{E''}^{E'},j_{B''}^{B'} ) \in \mathcal {P}$.
	\end{enumerate}
\end{proof}
\section{Category of chains from bundles}
Consider $\xi_i = E_i \stackrel{p_i}{\rightarrow}  B_i$ of triples $(E_i, p_i, B_i)$ where $p_i$ is a map, $B_i$ is the base space and $E_i$ is the total space. Then the category $\textbf{Bun} $ of bundles is the category whose objects and morphisms are bundles $\xi_i = E_i \stackrel{p_i}{\rightarrow}  B_i$ and bundle morphisms $s_i = (u_i, f_i)$. One remarkable thing regarding the category $\textbf{Bun} $ is that there are many chain  categories associated with them: viz., chains of bundles, chains of fibres and chains of long exact sequences. 
\subsection{Chains of bundles}
A chain of bundles \index{chain of bundles} is an indexed family of bundles and bundle morphisms of the form $\cdots \rightarrow (E_i, p_i, B_i) \stackrel{s_i}{\rightarrow} (E_{i+1}, p_{i+1} ,B_{i+1})  \rightarrow \cdots$
where $(E_i, p_i, B_i)$ are bundles and $s_i = (u_i, f_i)$ is the bundle map between bundles, and the composition of morphisms is there wherever possible. It is easy to see that the chain of bundles is a subcategory of $\textbf{Bun}$.
\begin{exam}
	Consider the group $G =\mathbb{Z}$. Consider the normal series $\mathbb{Z} \supset 9\mathbb{Z} \supset 45 \mathbb{Z} \supset 180\mathbb{Z} \supset \{\textbf{0}\}$.
	$\mathbb{Z}$ is an $n\mathbb{Z}$ bundle over $\mathbb{Z}/n\mathbb{Z}$.	We have the following sequence of principal group bundles: 
	\begin{center} \small
		\begin{tikzcd} 
		\mathbb{Z}	\arrow{d} {j}  \arrow["p_1"]{r} & \mathbb{Z}/180\mathbb{Z} \arrow{d} {f_1}\\
		\mathbb{Z}	\arrow{d} {j}  \arrow["p_2"]{r} & \mathbb{Z}/45\mathbb{Z} \arrow{d} {f_2}\\
		\mathbb{Z}	\arrow{d} {j}  \arrow["p_3"]{r} & \mathbb{Z}/9\mathbb{Z} \arrow{d} {f_3}\\
		\mathbb{Z}	   \arrow["p_4"]{r} & \mathbb{Z}/\{\textbf{0}\} \\
		\end{tikzcd}
	\end{center}
	Here $j$ denotes inclusion,\\ $p_i: \mathbb{Z} \rightarrow \mathbb{Z}/n_i\mathbb{Z} $ is defined by $p_i(z) = z + n_i \mathbb{Z} $ and\\ $f_i: \mathbb{Z}/n\mathbb{Z} \rightarrow \mathbb{Z}/m\mathbb{Z} $ is defined by $f_i(z+ n \mathbb{Z}) = z+ m \mathbb{Z}$.
\end{exam}

Next, we describe the category of chains of bundles. We begin by defining morphisms between chains of bundles. Consider two chains of bundles $c $ and $d $ given below.\\
\begin{center}
	\begin{tikzcd} 
	\vdots  \arrow{d} &    \vdots \arrow{d} & & \vdots \arrow{d} &    \vdots \arrow{d}\\
	E_i \arrow{d} {u_i}  \arrow["p_i"]{r} & B_i \arrow{d} {f_i} & & E_i' \arrow{d} {u_i'}  \arrow["p_i'"]{r} & B_i' \arrow{d} {f_i'} \\
	E_{i+1} \arrow{d}{u_{i+1}}  \arrow["p_{i+1}"]{r} & B_{i+1} \arrow{d}{f_{i+1}} & & E_{i+1}' \arrow{d}{u_{i+1}'}  \arrow["p_{i+1}'"]{r} & B_{i+1}' \arrow{d}{f_{i+1}'} \\
	\vdots  &   \vdots & &\vdots  &   \vdots    \\
	\end{tikzcd}\\
	$c$ \hspace{4.2cm}  $d$\\
	
	A morphism between these two chains of bundles is a map $m : c \rightarrow d$ between them with $m= (g_i, h_i)$ where $g_i: E_i \rightarrow E_i'$ and $h_i: B_i \rightarrow B_i'$ 
	
	\begin{tikzcd}[row sep=1.5em, column sep = 1.5em]
	E_i \arrow[rr, dashed] \arrow[dr,swap,"g_i"] \arrow[dd,swap,"u_i"] &&
	B_i \arrow[dd, "f_i"] \arrow[dr,"h_i"] \\
	& E_i' \arrow[rr, dashed] \arrow[dd,"u_i'"] &&
	B_i' \arrow[dd,"f_i'"] \\
	E_{i+1} \arrow[rr,dashed] \arrow[dr, "g_{i+1}"] && B_{i+1} \arrow[dr, "h_{i+1}"] \\
	& E_{i+1}' \arrow[rr, dashed] && B_{i+1}'
	\end{tikzcd}
\end{center}
such that following conditions holds:\\
\begin{equation} \label{A}
g_iu_i' = u_i g_{i+1}\quad \text{and} \quad f_ih_{i+1} = h_i f_i' 
\end{equation}
It is easily seen that these conditions holds when $(g_i,f_i)$ are bundle morphisms between the bundles $E_i  \stackrel{p_i}{\rightarrow} B_i$ and $E_i' \stackrel{p_i'}{\rightarrow}  B_i'$ for all $i$.
With objects as sequences of bundles and morphisms between them as defined above and the composition of two composible morphisms  $m= (g_i, h_i): c \rightarrow d$ and $n= (g_i', h_i'): d \rightarrow e$ is between $c$ and $e$ given by $mn = (g_ig_i', h_ih_i')$. Using (\ref{A}) it can be seen that $mn$ is a morphism. The identity morphism is $(1_{E_i}, 1_{B_i})$ and all axioms of category holds good. Hence sequence of bundles with bundle morphisms between them forms a chain category. 

Now we define a subchain of a bundle of chains and see that the above chain category is a category with subobjects.
Consider two chains of bundles $c$ and $d$ where\\
 $$c := \cdots \stackrel{(u_{i-1}, f_{i-1})}{\rightarrow} (E_i, p_i, B_i)  \stackrel{(u_{i}, f_{i})}{\rightarrow}(E_{i+1}, p_{i+1}, B_{i+1})  \stackrel{(u_{i+1}, f_{i+1})}{\rightarrow} \cdots $$ and $$d := \cdots \stackrel{(u_{i-1}', f_{i-1}')}{\rightarrow} (E_i', p_i', B_i')  \stackrel{(u_{i}', f_{i}')}{\rightarrow}(E_{i+1}', p_{i+1}', B_{i+1}')  \stackrel{(u_{i+1}', f_{i+1}')}{\rightarrow}\cdots  $$
 Then $c$ is said to be a subchain of $d$ if $(E_i, p_i, B_i)$ is a subbundle of $(E_i', p_i', B_i')$ and the following two conditions holds for all $i$:
 $$ j_{E_i}^{E_i'} u_i' = u_ij_{E_{i+1}}^{E_{i+1}'} \quad \text{and} \quad f_i j_{B_{i+1}}^{B_{i+1}'} =  j_{B_i}^{B_i'}f_i'  $$.
 The inclusion morphism $j_c^d$ is given by $j_c^d = ( j_{E_i}^{E_i'}, j_{B_i}^{B_i'} )$.\\
 It can be seen that above two conditions holds if the bundle maps $p_i'$ of chain $d$ are all monomorphisms and the bundle maps $p_i$ of chain $c$ are all epimorphisms.
 Thus it can be seen that the category of chains of bundles forms a category with subobjects where subobjects are precisely subchains defined as above.\\
 

Recall a fibre bundle (see \cite{Saun}). 	For manifolds $E$ and $M$, let $\pi: E \rightarrow M $ be a surjective submersion, then the triple $(E, \pi, M)$ is known as a  fibred manifold. A locally trivial fibred manifold will be a fibre bundle. A  section \index{section} of the bundle $(E, \pi, M)$ is a map $\phi: M \rightarrow E$ satisfying $\phi \circ \pi = id_M$. A local section \index{local section} of the bundle is a section of the subbundle $(\pi^{-1}(U), \pi|_{\pi^{-1}(U)}, U)$, where $U$ is a submanifold of $M$. Consider bundles $(E,\pi, M)$ and $(F,\rho, N)$ and let $\phi: M \rightarrow E$ be a section of $\pi$. If $(f,\bar{f})$ is a bundle morphism from  $(E,\pi, M)$ to $(F,\rho, N)$ with additional property that it projects to diffeomorphism. i.e., $\bar{f}: M \rightarrow N$ is a diffeomorphism, then $\tilde{f}(\phi) = \bar{f}^{-1} \circ \phi \circ f$ is a section of $\rho$.
It can be seen that collection of all fibred manifolds satisfying local triviality condition and all bundle morphisms between such bundles which projects to a diffeomorphism forms a category.\\
Next, we present an example of a category of chains of bundles, namely the category of chains of jet bundles. The relevant notions are taken from \cite{Saun}.
\begin{exam}
	Let $(E,\pi, M)$ be a bundle. Denote $\Gamma(\pi)=$ set of all sections of $\pi$ and $\Gamma(\pi)_x=$ set of all local sections of $\pi$ whose domain include $ x \in M$.
	Let $\phi \in \Gamma(\pi)_x$ and $k \in \mathbb{N}$. $\psi \in \Gamma(\pi)_x$ is said to be equivalent to $\phi$ to order $k$ if for every $f \in C^{\infty}(E)$ and every smooth curve $\gamma: \mathbb{R} \rightarrow M$ satisfying $\gamma(0) = x$,
	$$\frac{d^r}{dt^r}_{|t=0} (t \mapsto \gamma \circ \phi \circ f (t)) = \frac{d^r}{dt^r}_{|t=0} (t \mapsto \gamma \circ \psi \circ f (t)) $$ whenever $0 \leq r \leq k$.
	\par{}
	The $k-$ jet of $\phi$ at $x$ is the set \index{$j_x^k\phi$}$$ j_x^k\phi = \{\psi \in \Gamma(\pi)_x| \, \psi \, \text{is equivalent to} \, \phi \, \text{to order}\, k \}.$$

	\index{$J^k(\pi)$}	$$ J^k(\pi) = \{ j_x^k\phi | \, \phi \in \Gamma(\pi)_x, x \in dom(\phi) \}$$ and 
	\index{$J_x^k(\pi)$}	$$J_x^k(\pi) = \{ j_x^k\phi | \, \phi \in \Gamma(\pi)_x \}$$
	Then \[ J^k(\pi) = \bigsqcup_{x \in M} J^k_x(\pi)\]

	Let $\pi_k^{k-1}: J^k(\pi) \rightarrow J^{k-1}(\pi)$ by $\pi_k^{k-1}(j_x^k\phi) = j_x^{k-1} \phi$. For $k-1 > l, l \geq 0$, define $\pi_k^l: J^k(\pi) \rightarrow J^l(\pi)$ recursively by $$ \pi_k^l = \pi_k^{k-1} \circ \pi_{k-1}^l.$$ For $k > l \geq 0, (J^k(\pi), \pi_k^l, J^l(\pi))$ is a bundle. 	Note that $J^0(\pi) = E$ and $\pi_k^0: J^k(\pi) \rightarrow J^0(\pi) = E$.\\
	Define $\pi_k: J^k(\pi) \rightarrow M$ by $\pi_k = \pi_k^0 \circ \pi$. Then $(J^k(\pi), \pi_k, M)$ is a fibre bundle where $\pi_k^{-1}(x) = J_x^k(\pi)$ and is termed as a \index{jet bundle} jet bundle. For the bundles $(E, \pi, M)$ and $ (F, \rho, N)$  and the bundle morphism $(f,\bar{f})$  where $\bar{f}$ is a diffeomorphism, the $k^{th}$ prolongation of $f$ is the map $j^kf = j^k(f,\bar{f}): J^k(\pi) \rightarrow J^k(\rho)$ defined by $$j^kf(j^k_x \phi) = j^k(f,\bar{f})(j^k_x \phi) = j^k_{\bar{f}(x)}\tilde{f}(\phi).$$
	The following are bundle morphisms:\\
	$$(j^kf,f): (J^k(\pi), \pi_k^0, E) \rightarrow (J^k(\rho), \rho_k^0, F), $$  $$(j^kf,\bar{f}): (J^k(\pi), \pi_k, M) \rightarrow (J^k(\rho), \rho_k, N), \, \text{and}$$   $$(j^kf,j^{k-1}f): (J^k(\pi), \pi_k^{k-1}, J^{k-1}(\pi)) \rightarrow (J^k(\rho), \rho_k^{k-1},J^{k-1}(\rho))$$ 
	
	Now, consider the bundle maps $f: (E, \pi, M) \rightarrow (F, \rho, N) \, \text{and} \, g: (F, \rho, N) \rightarrow (K, \sigma, L)$  which projects to diffeomorphisms. Then
	$$j^k(f \circ g, \overline{f \circ g}) = j^k(f, \bar{f}) \circ j^k(g, \bar{g})$$ and $$ j^k(id_E,id_M) = id_{J^k(\pi)}.$$
	\\  
	We get a sequence of jet bundles from the locally trivial fibred manifold $(E, \pi, M)$ as follows:
	$$\cdots \rightarrow J^k(\pi) \stackrel{\pi_k^{k-1}}{\rightarrow}
	J^{k-1}(\pi) \stackrel{\pi_{k-1}^{k-2}}{\rightarrow} \cdots \stackrel{\pi_{2}^{1}}{\rightarrow}J^1(\pi) \stackrel{\pi_{1}^{0}}{\rightarrow} J^0(\pi) = E  \stackrel{\pi}{\rightarrow} M $$
	In particular we illustrate above example when $M = \mathbb{R}^2, E = M \times \mathbb{R}$. Consider the trivial bundle $(E, \pi, M)$. Sections of $\pi$ are just smooth functions $f:\mathbb{R}^2 \rightarrow\mathbb{R}$.\\
	A section $ u = f(x,y) \in C^{\infty}(M)$ has derivatives of
	\begin{enumerate}
		\item zeroth order: $u = f(x,y)$
		\item first order: $u_x, u_y$
		\item second order: $u_{xx}, u_{yy},u_{xy}.$
	\end{enumerate}
	If $(x,y) $ are coordinates in $M$, then\\
	\begin{itemize}
		\item	E has coordinates: $(x,y,u)$\\
		\item	$J^1(\pi)$ has coordinates: $(x,y,u,u_x,u_y)$  and \\
		\item	$J^2(\pi)$ has coordinates: $(x,y,u,u_x,u_y,u_{xx}, u_{yy},u_{xy})$.\\
	\end{itemize}
	We have the chain $$ J^2(\pi) \stackrel{\pi_2^1}{\rightarrow} J^1(\pi)  \stackrel{\pi_1^0}{\rightarrow} E  \stackrel{\pi}{\rightarrow} M$$
	$$ \text{where} \quad
	\pi_2^1(x,y,u,u_x,u_y,u_{xx}, u_{yy},u_{xy}) = (x,y,u,u_x,u_y) \,\,\, \text{and}$$ $$
	\pi_1^0(x,y,u,u_x,u_y) = (x,y,u).$$
	
	If $\cdots \rightarrow J^k(\pi) \stackrel{\pi_k^{k-1}}{\rightarrow}
	J^{k-1}(\pi) \stackrel{\pi_{k-1}^{k-2}}{\rightarrow} \cdots \stackrel{\pi_{2}^{1}}{\rightarrow}J^1(\pi) \stackrel{\pi_{1}^{0}}{\rightarrow} J^0(\pi) = E  \stackrel{\pi}{\rightarrow} M $ and $\cdots \rightarrow J^k(\rho) \stackrel{\rho_k^{k-1}}{\rightarrow}
	J^{k-1}(\rho) \stackrel{\rho_{k-1}^{k-2}}{\rightarrow} \cdots \stackrel{\rho_{2}^{1}}{\rightarrow}J^1(\rho) \stackrel{\rho_{1}^{0}}{\rightarrow} J^0(\rho) = F  \stackrel{\rho}{\rightarrow} N $ are two chains of jet bundles and if there is bundle morphism $(f, \bar{f})$ that project to diffeomorphism between $(E, \pi,M)$ and $(F, \rho, N)$, then we get a chain map between  above two chains of bundles as follows:\\
	\begin{center}
		\begin{tikzcd}  
		\cdots \arrow[" "]{r} & J^k(\pi)\arrow["\pi_k^{k-1}"]{r} \arrow{d}{j^kf} & J^{k-1}(\pi) \arrow["\pi_{k-1}^{k-2}"]{r} \arrow{d}{j^{k-1}f} & \cdots \arrow["\pi_1^0" ]{r} & E \arrow["\pi"]{r} \arrow{d}{f} & M \arrow{d}{\bar{f}}\\
		\cdots \arrow[" " ]{r} & J^k(\rho)\arrow["\rho_k^{k-1}"]{r}   & J^{k-1}(\rho) \arrow["\rho_{k-1}^{k-2}"]{r}   & \cdots \arrow["\rho_1^0" ]{r} & F \arrow["\rho"]{r}   & N  
		\end{tikzcd}
	\end{center}  
	Based on previous discussions, it can be seen that composition of two such chain maps is possible and the collection of chains of jet bundles with morphisms as chain map defined above form a category of chains of bundles.\\
\end{exam}
\subsection{Chains of fibres}
Consider the bundle $p: E \rightarrow B$ which consists of a neighbourhood $U$ at each point of $B$ for which a homeomorphism $h: p^{-1}(U) \rightarrow U \times F$ can be defined such that the following diagram commutes.
\begin{center}  
	\begin{tikzcd} 
	p{^{-1}(U)} \arrow["h "]{r} \arrow["p"]{d} & U \times F  \arrow{dl}{}
	\\
	U \\
	\end{tikzcd}		
\end{center}
See \cite{Dale}. Then,  $p^{-1}(U) \rightarrow F $ which is a homeomorphism on each fibre and $p: E \rightarrow B$, is a \index{fibre bundle} fibre bundle. Consider subbundles of this fibre bundle, and by taking fibres of this bundle and its subbundles as the objects and morphisms are the maps between fibres that preserves structure of fibres we obtain the category of fibres  associated with this fibre bundle and it is denoted by $\mathcal{FC}_p $. If $(E', p', B')$ is a subbundle of $(E, p, B)$, then we have $E' \subseteq E$, $B' \subseteq B$ and $p' = p|_{E'} $. It can be seen that $(p')^{-1}(b) \subseteq p^{-1}(b) $. Thus if $(E', p', B')$ is a subbundle of $(E, p, B)$ then fibres of $(E', p', B')$ will be subfibres of $(E, p, B)$. Hence $\mathcal{FC}_p $ is a category with subobjects where subfibres are the subobjects.

Consider $(E_1, p_1, B_1) \subset (E_2, p_2, B_2) \subset \cdots \subset (E_n, p_n, B_n) = (E, p, B) $ be a chain of subbundles in \textbf{Bun} where $B_1 \subset B_2 \subset \cdots \subset B_n =B$. For each $b \in B_1$ we can form a chain $$ p_1^{-1}(b) \stackrel{s_1} \rightarrow p_2^{-1}(b) \stackrel{s_2} \rightarrow \cdots \stackrel{s_{n-1}} \rightarrow  p_n^{-1}(b)= p^{-1}{(b)}$$ where $s_i$ is a morphism in the homset $Hom(p_i^{-1}(b), p_{i+1}^{-1}(b))$ in the category $\mathcal{FC}_p $ for all $i$, which also include identity morphism on $p_i^{-1}(b)$ and all possible composite of morphisms. Now it is easy to see that these \index{chains of fibres} chains of fibres forms  a category of chains  in $\mathcal{FC}_p $. But in this category it is not always possible to define subobject relation because chains of fibres corresponding to different base points are not related by canonical inclusion maps, since fibres over distinct points can be disjoint and admit no natural inclusion between them.
\begin{exam}
	We may consider a vector bundle $p: E \rightarrow B$. Then, in the category $\mathcal{FC}_p $, objects are fibres of the vector bundle and its subbundles and morphisms are structure preserving maps between them. If dimension of $B$ is $n$, then the objects are vector spaces of dimension less than or equal to $n$ and morphisms are linear maps between them.   
	A chain in this case will be a sequence of vector spaces and $s_i$'s are linear maps between the corresponding fibres. \\
\end{exam}
\subsection{Category of long exact sequences arising from fibrations}
Long exact sequences plays a central role in algebraic topology and homological algebra, enabling a way to systematically analyze and relate algebraic invariants. They occur naturally in various contexts, such as chain complexes, fibrations and cohomology theories. \\
Here we see long exact sequences arising from a fibre bundle and recall some basic notions and results regarding this.

\begin{defn}
	An exact sequence \index{exact sequence} of groups is a sequence of homomorphisms
	$$ \cdots  \longrightarrow G_{n+1}  \stackrel{\partial_{n+1}}{\longrightarrow} G_n  \stackrel{\partial_{n}}{\longrightarrow} G_{n-1}  \stackrel{\partial_{n-1}}{\longrightarrow} \cdots$$ 
	such that $Ker \, \partial_{n} = Im \, \partial_{n+1} $ for each $n$.
	
\end{defn} 
The sequence of groups and homomorphisms may be either finite or infinite. An infinite exact sequence is called a long exact sequence where the indices typically represent some geometric or algebraic structure. These sequences often arise from constructions like the homology or cohomology of spaces, extensions of groups and fibrations.
\begin{defn}[cf.\cite{Allen}]
	Let $I^n$ be the $n$ dimensional unit cube. For a space $X$ with base point $x_0 \in X$, $\pi_n(X,x_0)$ \index{$\pi_n(X,x_0)$} is the set of homotopy classes of maps $f: (I^n, \partial I^n) \rightarrow (X,x_0)$ where homotopies $f_t$ are required to satisfy $f_t(\partial I^n)= x_0$ for all $t$.
\end{defn}
Maps $(I^n, \partial I^n) \rightarrow (X,x_0)$ are the same as maps of the quotient $I^n/ \partial I^n= S^n$ to $X$ taking the base point $s_0 =  \partial I^n/ \partial I^n$ to $x_0$. Thus $\pi_n(X,x_0)$ can be viewed as the set of homotopy classes of maps $(S^n, s_0) \rightarrow (X, x_0)$, with homotopies taken through maps of the same type  $(S^n, s_0) \rightarrow (X, x_0)$.

\begin{defn}[cf.\cite{Allen}]
	A map $p: E \rightarrow B$ is said to have homotopy lifting property \index{homotopy lifting property} with respect to a space $X$, if given a homotopy $g_t : X \rightarrow B$ and a map $ \tilde{g}_0 : X \rightarrow E$ lifting $g_0$, so $ \tilde{g}_0 p= g_0$ then there exists a homotopy $ \tilde{g}_t : X \rightarrow E$ lifting $g_t$.
\end{defn}
\begin{defn}[cf.\cite{Allen}]
	A fibration is a map $p: E \rightarrow B$ having homotopy lifting property with respect to all spaces $X$.
\end{defn}

\begin{prop}[cf.\cite{Allen}]
	A fibre bundle $ p:E \rightarrow B$ has the homotopy lifting property with respect to all $CW$ pairs $(X, A)$. i.e., for any homotopy $g_t: X \rightarrow B$, given any lift $\tilde{g}_0 :A \rightarrow E$ of $g_0 : A \rightarrow B$ there exists lifted homotopy $\tilde{g}_t : X \rightarrow E$ such that $\tilde{g}_t \circ p = g_t$ and $\tilde{g}_{t|_A} = \tilde{g}_0, \forall t \in [0,1].$
\end{prop}

A standard long exact sequence in homotopy theory comes from a fibration $F\stackrel{j}{\rightarrow} E \stackrel{p}{\rightarrow} B$, where $F$ is the fibre, $E$ is the total space, and $B$ is the base space. The associated sequence of homotopy groups is:
$$ \cdots  \longrightarrow \pi_{n+1}(B)   \longrightarrow \pi_n(F)   \stackrel{j^*}\longrightarrow \pi_n(E) \stackrel{p^*} \longrightarrow \pi_n(B) \stackrel{\partial} \longrightarrow \pi_{n-1}(F) \cdots$$ 
An important factor of long exact sequences is the presence of boundary maps, which connect different levels of the sequence. For instance, in the long exact sequence of a fibration, the boundary map $\partial : \pi_{n}(B)   \longrightarrow \pi_{n-1}(F)$ arises from the homotopy lifting property of the fibration.
We see how $\partial : \pi_n(B) \rightarrow \pi_{n-1}(F)$ is defined. 

Consider $[f] \in \pi_n(B)$ and $b_0 \in B$ such that $F= \pi^{-1}(b_0)$ and let $e_0 \in F$.  Then $f: S^n \rightarrow B$ which can be viewed as the map $f: (D^n, \partial D^n) \rightarrow B$ with $f(\partial D^n)= b_0$. Consider the homotopy $H(x,t):D^n \rightarrow B $ where $ H(x,t) = f(x) \quad \forall t$. Now $f_{|_{\partial{D^n}}} : \partial{D^n} \rightarrow \{b_0\} \subset B$. Define $G(x,t): \partial{D^n} \rightarrow  B$ by $g_t(x) = b_0 \quad \forall x, \forall t$. Then  $g_t(x) = f_{|_{\partial{D^n}}}(x)  \quad \forall t $. Now define $\tilde{G}(x,t):\partial D^n \rightarrow F $ by $\tilde{G}(x,t) = e_0 \quad \forall x, \forall t$. Clearly $\tilde{G}(x,0)$ is an extension of ${G}(x,0)$. Then there exists lifted homotopy of $H$, $\tilde{H}(x,t): D^n \rightarrow E $  with $\tilde{H}(x,t)_{|_{\partial{D^n}}}(x) = \tilde{G}(x,0) = \tilde{G}(x,t)$ and $\tilde{H} \circ p = H$.\\
Thus for $[f] \in \pi_n(B)$ , define $\partial([f])= [\tilde{G}(x,t)] \in \pi_{n-1}(F)$.

The category of long exact sequences arising from fibrations is another category of chains associated with fibre bundles.
Next we describe the morphism between two such long exact sequences.

Let $F_1\stackrel{j_1}{\rightarrow} E_1 \stackrel{p_1}{\rightarrow} B_1$ and $F_2 \stackrel{j_2}{\rightarrow} E_2 \stackrel{p_2}{\rightarrow} B_2$ be two fibrations. Consider the following long exact sequences arising from these fibrations:
$$ \cdots  \longrightarrow \pi_{n+1}(B_1)   \longrightarrow \pi_n(F_1) \stackrel{j_1^*}{\longrightarrow} \pi_n(E_1)  \stackrel{p_1^*}{\longrightarrow}\pi_n(B_1) \stackrel{\partial_1}{\longrightarrow} \pi_{n-1}(F_1) \cdots$$
and
$$\cdots  \longrightarrow \pi_{n+1}(B_2)   \longrightarrow \pi_n(F_2)  \stackrel{j_2^*}{\longrightarrow} \pi_n(E_2) \stackrel{p_2^*}{\longrightarrow} \pi_n(B_2) \stackrel{\partial_2}{\longrightarrow} \pi_{n-1}(F_2) \cdots$$
where $j_i^*$ is the map induced by $j_i: F_i \rightarrow E_i$, $p_i^*$ is the map induced by the projection $p_i: E_i \rightarrow B_i$ and $\partial_i$ is the boundary map defined using lifting property of fibration  $F_i \stackrel{j_i}{\rightarrow} E_i \stackrel{p_i}{\rightarrow} B_i$ (see \cite{Allen}).
If there is a bundle map $\Psi:  (E_1,p_1,B_1,F_1) \rightarrow (E_2,p_2,B_2,F_2) $, then it induces a map $\Psi ^*$ between above two long exact sequences. The map between the two long exact sequences is seen as commutative diagram below: \\
\begin{center}
	\begin{tikzcd}
	\cdots     \arrow[" "]{r}  	& \pi_n(F_1) \arrow[" j_1^*"]{r} \arrow{d}[swap]{\Psi_F^*} & \pi_n(E_1) \arrow[" p_1^*"]{r} \arrow{d}{\Psi_E^*}& \pi_{n}(B_1)\arrow["\partial_1 "]{r} \arrow{d}[swap]{\Psi_B^*} & \pi_{n-1}(F_1) \arrow{d}{\Psi_F^*} \arrow[" "]{r}& \cdots\\
	\cdots   \arrow [" "]{r}  & \pi_n(F_2) \arrow[" j_2^*"]{r} &  \pi_n(E_2) \arrow[" p_2^*"]{r}& \pi_{n}(B_2) \arrow["\partial_2 "]{r} & \pi_{n-1}F_2)\arrow[" "]{r}  & \cdots   
	\end{tikzcd}
\end{center}
$\Psi_F^*, \Psi_E^* $ and $\Psi_B^*$ are the maps induced by $\Psi_F, \Psi_E $ and $\Psi_B$ respectively and the boundary maps $\partial_1$ and $\partial_2$ satisfy the compatibility condition: 
$$ \partial_1 \circ \Psi_F^* = \Psi_B^* \circ \partial_2.$$
Choosing long exact sequences associated with fibrations as objects and morphism between them induced by the bundle map, we obtain category of long exact sequences arising from fibrations.\\
Next we define subfibration and see that above category is a category with subobjects.  
Let $F\stackrel{j}{\rightarrow} E \stackrel{p}{\rightarrow} B$ be a fibration. Then $F'\stackrel{j'}{\rightarrow} E' \stackrel{p'}{\rightarrow} B'$ is a subfibration if $E' \subseteq E, B' \subseteq B, F' \subseteq F, $ and $ p' = p|_{E'}: E' \rightarrow B'$ is also a fibration. Then the long exact sequence corresponding to $F'\stackrel{j'}{\rightarrow} E'$ will be a subchain of long exact sequence corresponding to $F\stackrel{j}{\rightarrow} E \stackrel{p}{\rightarrow} B$. Choosing "being subchain" as the subobject relation, the category of long exact sequences arising from fibrations forms a category of chains with subobjects.
\\
  \textbf{Acknowledgement}
  The first author thank Council of Scientific and Industrial Research (CSIR), India for the support.

\end{document}